\documentclass[a4paper,10pt]{amsart}
\usepackage{amsmath,amsthm,amssymb,hyperref,graphicx,mathrsfs,mathtools}  
\newtheorem{theorem}{Theorem}[section]
\newtheorem{corollary}{Corollary}[section]
\newtheorem{lemma}{Lemma}[section]
\newtheorem{thmx}{Theorem}[section]
\newtheorem{remark}{Remark}[section]

\allowdisplaybreaks
\numberwithin{equation}{section}
\title{Zygmund-type inequalities for an operator preserving inequalities between polynomials}
\author{ Nisar. A. Rather, Suhail Gulzar and K.A. Thakur}
\begin{document}
\maketitle
\date{}
\begin{center} 
P.G.Department of Mathematics, Kashmir University, Hazratbal,
\end{center} 
\begin{center}
Srinagar-190006, India
\end{center}
\begin{center}
e-mail: sgmattoo@gmail.com
\end{center}
\begin{center} 
*Department of Mathematics, Degree College Ganderbal,
\end{center} 
\begin{center}
Ganderbal Kashmir- India
\end{center}
\begin{abstract}
 In this paper, we present certain new $L_p$ inequalities for $\mathcal B_{n}$-operators which include some known polynomial inequalities as special cases.
\end{abstract}
\section{Introduction and statement of results}
Let $\mathscr{P}_n$ denote the space of all complex polynomials $P(z)=\sum_{j=0}^{n}a_{j}{z}^{j}$ of degree $n$. For $P\in \mathscr{P}_n$, define
\[\left\|P(z)\right\|_{0}:=\exp\left\{\frac{1}{2\pi}\int_{0}^{2\pi}\log\left|
P(e^{i\theta})\right|d\theta\right\},\]
\[\left\|P(z)\right\|_{p}:=\left\{\frac{1}{2\pi}\int_{0}^{2\pi}\left|
P(e^{i\theta})\right|^{p}\right\}^{1/p},\,\, 1\leq p<\infty,\]
\[\left\|P(z)\right\|_{\infty}:=\underset{\left|z\right|=1}{\max}\left|P(z)\right|,\quad m(P,k):=\underset{\left|z\right|=k}{\min}\left|P(z)\right|,\,k>0\]
and denote for any complex function $\psi : \mathbb C\rightarrow \mathbb C$ the composite function of  
$P$ and $\psi$, defined by
$\left(P\circ\psi\right)(z):=P\left(\psi(z)\right)\,\,\,\,(z\in\mathbb C)$, as $P\circ\psi$.

If $P\in \mathscr{P}_n$, then
\begin{equation}\label{1}
\left\|P^{\prime}(z)\right\|_{p}\leq n\left\|P(z)\right\|_{p},\,\,\,\,\,\,p\geq 1
\end{equation} 
and
\begin{equation}\label{2}
\left\|P(Rz)\right\|_{p}\leq R^{n}\left\|P(z)\right\|_{p},\,\,\,R>1,\,\,\,\,p>0,
\end{equation}\\
Inequality \eqref{1} was found out by Zygmund \cite{18} whereas inequality \eqref{2} is a simple 
consequence of a result of Hardy \cite{8}. Arestov \cite{2} proved that \eqref{1} remains true for $0<p<1$ as well. For $p=\infty$, the inequality \eqref{1} is due to Bernstein (for reference, see \cite{11,15,16}) whereas the case $p=\infty$ of inequality \eqref{2} is a simple consequence of the maximum modulus principle ( see \cite{11,12,15}). Both the inequalities \eqref{1} and \eqref{2} can be sharpened if we restrict ourselves to the class of polynomials having no zero in $\left|z\right|<1.$ In fact, if $P\in \mathscr{P}_n$ and $ P(z) \neq 0$ in $\left|z\right|< 1$, then inequalities \eqref{1} and \eqref{2} can be respectively replaced by
\begin{equation}\label{3}
\left\|P^{\prime}(z)\right\|_{p}\leq n\frac{\left\|P(z)\right\|_{p}}{\left\|1+z\right\|_{p}},\,\,\,\,p\geq 0
\end{equation} 
and
\begin{equation}\label{4}
\left\|P(Rz)\right\|_{p}\leq \frac{\left\|R^{n}z+1\right\|_{p}}{\left\|1+z\right\|_{p}}\left\|P(z)\right\|_{p},\,\,\,R>1,\,\,\, p>0.
\end{equation}
Inequality \eqref{3} is due to De-Bruijn \cite{6}(see also \cite{3}) for $p\geq 1$. Rahman and Schmeisser \cite{1} extended it for $0<p<1,$ whereas the inequality \eqref{4} was proved by Boas and Rahman \cite{5} for $p\geq 1$ and later it was extended for $0<p<1$ by Rahman and Schmeisser \cite{14}. For $p = \infty$, the inequality \eqref{3} was conjectured by Erd\"{o}s and later verified by Lax \cite{9} whereas inequality \eqref{4} was proved by Ankeny and Rivlin \cite{1}.

As a compact generalization of inequalities \eqref{3} and \eqref{5}, Aziz and Rather \cite{4} proved that if $P\in \mathscr{P}_n$ and $P(z)$ does not vanish in $|z|<1,$ then for $\alpha,\beta\in\mathbb{C}$ with $|\alpha|\leq 1,$ $|\beta|\leq 1,$ $R>r \geq 1$ and $p>0$,
\begin{equation}\label{5}
\left\|P(Rz)+\phi_{n}\left(R, r, \alpha, \beta\right) P(rz)\right\|_{p}\leq\dfrac{C_p}{\|1+z\|_p} \left\|P(z)\right\|_{p}
\end{equation}
where
\begin{equation}\label{6}
C_p=\left\|\left(R^n+\phi_{n}\left(R, r, \alpha, \beta\right)r^n\right)z+(1+\phi_n(R,r,\alpha,\beta))\right\|_p
\end{equation}
and
\begin{equation}\label{phi}
\phi_{n}\left(R, r, \alpha, \beta\right)_n=\beta\left\{\left(\dfrac{R+1}{r+1}\right)^n-|\alpha|\right\}-\alpha.
\end{equation}
If we take $\beta=0,\,\alpha=1$ and $r=1$ in \eqref{5} and divide two sides of \eqref{5} by $R-1$ then make $R\rightarrow 1,$ we obtain inequality \eqref{3}. Whereas inequality \eqref{4} is obtained from \eqref{5} by taking $\alpha=\beta=0.$

Rahman \cite{13} (see also Rahman and Schmeisser \cite[p. 538]{15}) introduced a class $\mathcal B_n$ of operators $B$ that maps $ P \in\mathscr{P}_n$ into itself. That is, the operator $B$ carries $P \in \mathscr{P}_n$ into
\begin{equation}\label{8}
B[P](z):= \lambda_{0}P(z)+\lambda_{1}\left(\frac{nz}{2}\right)\frac{P^{\prime}(z)}{1!}+\lambda_{2}\left(\frac{nz}{2}\right)^{2}\frac{P^{\prime\prime}(z)}{2!}
\end{equation}\label{9}
where $\lambda_{0},\lambda_{1}$ and $\lambda_{2}$ are such that all the zeros of 
\begin{equation}
u(z):= \lambda_{0}+C(n,1)\lambda_{1}z+C(n,2)\lambda_{2}z^{2},\,\,\,C(n,r)= n!/r!(n-r)!,
\end{equation}
lie in the half plane
\begin{equation}\label{10}
|z|\leq |z-n/2|
\end{equation}
and proved that if $P\in \mathscr{P}_n$ and $P(z)$ does not vanish in $|z|<1$, then
\begin{equation}\label{11}
\left|B[P\circ\sigma](z)\right|\leq \frac{1}{2}\left\{R^{n}\left|\Lambda_n\right|+\left|\lambda_{0}\right|\right\}\left\|P(z)\right\|_{\infty} \,\,\,\mbox{} for\,\mbox{}\,\,\,|z|= 1,
\end{equation}
(see \cite[Inequalities (5.2) and (5.3)]{13}) where $\sigma(z)=Rz$, $R\geq1$ and 
\begin{equation}\label{12}
\Lambda_{n}:=\lambda_{0}+\lambda_{1}\frac{n^{2}}{2} +\lambda_{2}\frac{n^{3}(n-1)}{8}.
\end{equation}
 As an extension of inequality \eqref{11} to $L_p$-norm, recently W.M. Shah and A. Liman \cite{17} 
while seeking the desired extension, they \cite[Theorem 2]{17} have made an incomplete attempt by claiming to have proved that
if $P\in \mathscr{P}_n$ and $P(z)$ does not vanish in $\left|z\right|<1$, then for each $R\geq 1$ and $p \geq 1$,
\begin{equation} \label{13}
\left\|B[P\circ\sigma](z)\right\|_p \leq \frac{R^{n}|\Lambda_{n}|+|\lambda_{0}|}{\left\|1+z\right\|_{p}}\left\|P(z)\right\|_p .
\end{equation}
where $B\in B_{n}$ and $\sigma(z)= Rz$ and $\Lambda_n$ is defined by \ref{12}.

Rather and Shah \cite{rs} pointed an error in the proof of \eqref{13}, they not only provided a correct proof but also extended it for $0\leq p<1$ as well. They proved:
\begin{thmx}\label{ta}
\textit {If $P\in \mathscr{P}_n$ and $P(z)$ does not vanish for $|z|<1,$ then for $0\leq p<\infty$ and $R> 1,$
\begin{equation}\label{tae}
\left\|B[P\circ\sigma](z)\right\|_p\leq \dfrac{\left\|R^n\Lambda_nz+\lambda_0\right\|_p}{\|1+z\|_p}\left\|P(z)\right\|_p,
\end{equation}
 $B\in \mathcal{B}_n,$ $\sigma(z)=Rz$ and $\Lambda_n$ is defined by \eqref{12}. The result is sharp as shown by $P(z)=az^n+b,$ $|a|=|b|=1.$}
\end{thmx}
Recently, Rather and Suhail Gulzar \cite{rs} obtained the following result which is a generalization of Theorem \ref{ta}.
\begin{thmx}\label{tb}
If $P\in \mathscr{P}_n$ and $P(z)$ does not vanish for $|z|<1,$ then for $\alpha\in \mathbb{C}$ with $|\alpha|\leq 1,$ $0\leq p<\infty$ and $R>1,$
\begin{align}\label{tbe}
\left\|B[P\circ\sigma](z)-\alpha B[P](z)\right\|_p\leq \dfrac{\left\|(R^n-\alpha)\Lambda_nz+(1-\alpha)\lambda_0\right\|_p}{\|1+z\|_p}\left\|P(z)\right\|_p,
\end{align}
where
 $B\in \mathcal{B}_n,$ $\sigma(z)=Rz$ and $\Lambda_n$ is defined by \eqref{12}. The result is best possible and equality in \eqref{tbe} holds for $P(z)=az^n+b,$ $|a|=|b|=1.$
\end{thmx}
If we take $\alpha=0$ in Theorem \ref{tb}, we obtain Theorem \ref{ta}.

In this paper, we investigating the dependence of
$$ \left\|B[P\circ\sigma](z) + \phi_{n}\left(R, r, \alpha, \beta\right) B[P\circ\rho](z)\right\|_{p} $$
on $\left\|P(z)\right\|_p$ for $\alpha$, $\beta\in\mathbb{C}$ with $|\alpha| \leq 1$,\,\,$|\beta| \leq 1$, $R >r \geq1$, $0\leq p <\infty$, 
$\sigma(z):=Rz$, $\rho(z):=rz$ and
$\phi_{n}\left(R, r, \alpha, \beta\right)$ is given by \eqref{phi}
and establish certain generalized $L_p$-mean extensions of the inequality \eqref{11} for $0\leq p <\infty$ and also a generalization of \eqref{5}.  In this direction, we first  present the following result which is a compact generalization of the inequalities \eqref{3}, \eqref{4}, \eqref{5} and \eqref{11} for $0 \leq p < 1$ as well.
\begin{theorem}\label{t1}
If $P\in \mathscr{P}_n$ and $P(z)$ does not vanish in $\left|z\right|<1$, then for then for $\alpha,\beta\in\mathbb{C}$ with $|\alpha|\leq1$, $|\beta|\leq1$, $R>r\geq1$ and $0\leq p <\infty $,
\begin{align}\label{t1e}\nonumber
&\left\|B[P\circ\sigma](z)+\phi_{n}(R,r,\alpha,\beta) B[P\circ\rho](z)\right\|_p\\&\qquad\qquad\qquad\leq \frac{ \left\|\left(R^{n}+\phi_{n}(R,r,\alpha,\beta)r^{n}\right)\Lambda_{n}z+\left(1+\phi_{n}(R,r,\alpha,\beta)\right)\lambda_{0}\right\|_{p}}{\left\|1+z\right\|_{p}}\left\|P(z)\right\|_p 
\end{align}
where $B \in\mathcal B_n$, $\sigma(z):=Rz$, $\rho(z):=rz$, $\Lambda_n$ and
$\phi_{n}\left(R, r, \alpha, \beta\right)$ are defined by\eqref{phi} and \eqref{12} respectively.
The result is best possible and equality in \eqref{t1e} holds for $P(z)=az^{n}+b,|a|=|b|\neq0$
\end{theorem}
\begin{remark}
\textnormal{If we take $\lambda_1=\lambda_2=0$ in \eqref{t1e}, we obtain inequality \eqref{5}.}
\end{remark}
For $\beta=0,$ inequality \eqref{t1e} reduces the following result.
\begin{corollary}\label{c1}
\textit{If} $P\in \mathscr{P}_n$ \textit{and} $P(z)$ \textit{does not vanish in} $\left|z\right|<1$, \textit{then for every real or complex number} $\alpha$ with $|\alpha|\leq1$, $R > r \geq 1$ \textit{and} $0\leq p <\infty$,
\begin{equation}\label{c1e}
\left\|B[P\circ\sigma](z)- \alpha B[P\circ\rho](z)\right\|_p\leq \frac{\left\|(R^{n}-\alpha r^{n})\Lambda_{n}z+(1-\alpha)\lambda_{0}\right\|_{p}}{\left\|1+z\right\|_{p}}\left\|P(z)\right\|_p 
\end{equation}
where  $B \in\mathcal B_n$, $\sigma(z):=Rz$, $\rho(z):=rz$ and $\Lambda_n$ is defined by \eqref{12}. The result is best possible and equality in \eqref{c1e} holds for $P(z)=az^{n}+b,\,\,\,|a|=|b|\neq0$.
\end{corollary}
\begin{remark}
\textnormal{For taking $\alpha=0$ in \eqref{c1e}, we obtain Theorem \eqref{ta} and for $r=1$ in \eqref{c1e}, we get Theorem \ref{tb}.}
\end{remark}
Instead of proving Theorem \ref{t1}, we prove the following more general result which includes Theorem \ref{t1} as a special case.
\begin{theorem}\label{t2}
If $P\in \mathscr{P}_n$ and $P(z)$ does not vanish in $\left|z\right|<1$, then for then for $\alpha,\beta,\delta\in\mathbb{C}$ with $|\alpha|\leq1$, $|\beta|\leq1, $ $|\delta|\leq 1,$ $R>r\geq1$ and $0\leq p <\infty $,
\begin{align}\nonumber\label{t2e}
 \Bigg\|B[P\circ\sigma](z)&+ \phi_{n}\left(R, r, \alpha, \beta\right)B[P\circ\rho](z)\\\nonumber+\delta&\dfrac{\Big(| R^n+\phi_{n}\left(R, r, \alpha, \beta\right) r^n||\Lambda_n|-|1+\phi_{n}\left(R, r, \alpha, \beta\right)||\lambda_0|\Big)m}{2}\Bigg\|_p\\
   &\leq \frac{ \left\|\left(R^{n}+\phi_{n}(R,r,\alpha,\beta)r^{n}\right)\Lambda_{n}z+\left(1+\phi_{n}(R,r,\alpha,\beta)\right)\lambda_{0}\right\|_{p}}{\left\|1+z\right\|_{p}}\left\|P(z)\right\|_p 
\end{align}
where $B \in B_n$, $\sigma(z):=Rz$, $\rho(z):=rz$, $\Lambda_n$ and
$\phi_{n}\left(R, r, \alpha, \beta\right)$ are defined by\eqref{phi} and \eqref{12} respectively.
The result is best possible and equality in \eqref{t1e} holds for $P(z)=az^{n}+b,|a|=|b|\neq0.$
\end{theorem}
\begin{remark}
\textnormal{For $\delta=0$ in \eqref{t2e}, we get Theorem \ref{t1}.}
\end{remark}
Next, corollary which is a generalization of \eqref{5} follows by taking $\lambda_1=\lambda_2=0$ in \eqref{t2e}.
\begin{corollary}\label{c2}
If $P\in \mathscr{P}_n$ and $P(z)$ does not vanish in $\left|z\right|<1$, then for then for $\alpha,\beta,\delta\in\mathbb{C}$ with $|\alpha|\leq1$, $|\beta|\leq1, $ $|\delta|\leq 1,$ $R>r\geq1$ and $0\leq p <\infty $,
\begin{align}\nonumber\label{c2e}
 \Bigg\|P(Rz)&+ \phi_{n}\left(R, r, \alpha, \beta\right)P(rz)\\\nonumber&+\delta\dfrac{\Big(| R^n+\phi_{n}\left(R, r, \alpha, \beta\right) r^n||-|1+\phi_{n}\left(R, r, \alpha, \beta\right)|\Big)m}{2}\Bigg\|_p\\
   &\qquad\leq \frac{ \left\|\left(R^{n}+\phi_{n}(R,r,\alpha,\beta)r^{n}\right)z+\left(1+\phi_{n}(R,r,\alpha,\beta)\right)\right\|_{p}}{\left\|1+z\right\|_{p}}\left\|P(z)\right\|_p 
\end{align}
where 
$\phi_{n}\left(R, r, \alpha, \beta\right)$ is defined by\eqref{phi}.
The result is best possible and equality in \eqref{c2e} holds for $P(z)=az^{n}+b,|a|=|b|\neq0.$
\end{corollary}

\section{Lemmas}
For the proofs of these theorems, we need the following lemmas. The first Lemma is easy to prove.
\begin{lemma}\label{l1}
\textit{If} $P\in \mathscr P_{n}$\textit{ and} $P(z)$ \textit{has all its zeros in} $\left|z\right|\leq 1$, \textit{ then for every} $R\geq r\geq 1$ \textit{and} $\left|z\right|=1$,
\begin{equation*}
\left|P(Rz)\right|\geq \left(\frac{R+1}{r+1}\right)^{n}\left|P(rz)\right|.
\end{equation*}
\end{lemma}
The following Lemma follows from  \cite[Corollary 18.3, p. 65]{10}.
\begin{lemma}\label{l2}
\textit{ If all the zeros of  polynomial} $P\in \mathscr{P}_n$ \textit{lie in} $\left|z\right|\leq 1$,\textit{ then all the zeros of the polynomial} $B[P](z)$\textit{ also lie in} $\left|z\right|\leq 1$.
\end{lemma} 
\begin{lemma}\label{l3}
If $F\in\mathscr P_{n}$ \textit{ has all its zeros in} $\left|z\right|\leq1$ \textit{and} $P(z)$ \textit{ is a polynomial of degree at most} $n$ \textit{such that}
\[|P(z)|\leq |F(z)|\,\,\, \textrm{for} \,\,\, |z|=1,\]
\textit{then for every} $\alpha, \beta\in\mathbb{C}$ \textit{with} $|\alpha|\leq 1$, $|\beta|\leq 1$, $ R \geq r\geq 1$, \textit{and} $|z|\geq 1$,
\begin{equation}\label{l3e}
|B[P\circ\sigma](z)+\phi_{n}\left(R, r, \alpha, \beta\right)B[P\circ\rho](z)|\leq |B[F\circ\sigma](z)+\phi_{n}\left(R, r, \alpha, \beta\right)B[F\circ\rho](z)| 
\end{equation}
where $B\in\mathcal B_n$, $\sigma(z):=Rz$, $\rho(z):=rz$, $\Lambda_n$ \textit{and}
$\phi_{n}\left(R, r, \alpha, \beta\right)$ \textit{are defined by} \eqref{12} \textit{and} \eqref{phi} \textit{respectively}.
\end{lemma}
\begin{proof}
 Since the polynomial $F(z)$ of degree $n$  has all its zeros in $|z|\leq 1 $ and $P(z)$ is a polynomial of degree at most $n$ such that 
\begin{equation}\label{l3e1}
|P(z)|\leq |F(z)|\,\,\,\, \textrm{for} \,\,\,\,|z|= 1,
\end{equation}
therefore, if $F(z)$ has a zero of multiplicity $s$ at $z=e^{i\theta_{0}}$, then $P(z)$ has a zero of multiplicity at least $s$ at $z=e^{i\theta_{0}}$. If $P(z)/F(z)$ is a constant, then the inequality \eqref{l3e} is obvious. We now assume that $P(z)/F(z)$ is not a constant, so that by the maximum modulus principle, it follows that
\[|P(z)|<|F(z)|\,\,\,\textrm{for}\,\, |z|>1\,\,.\]
Suppose $F(z)$ has $m$ zeros on $|z|=1$ where $0\leq m \leq n$, so that we can write
\[F(z) = F_{1}(z)F_{2}(z)\]
where $F_{1}(z)$ is a polynomial of degree $m$ whose all zeros lie on $|z|=1$ and $F_{2}(z)$ is a polynomial of degree exactly $n-m$ having all its zeros in $|z|<1$. This implies with the help of inequality \eqref{l3e1} that
\[P(z) = P_{1}(z)F_{1}(z)\]
where $P_{1}(z)$ is a polynomial of degree at most $n-m$. Now, from inequality \eqref{l3e1}, we get
\[|P_{1}(z)| \leq |F_{2}(z)|\,\,\,\textrm{for} \,\, |z|=1\,\]
where $F_{2}(z) \neq 0 \,\, for\,\, |z|=1$. Therefore
 for every $\lambda\in\mathbb{C} $ with $|\lambda|>1$, a direct application of Rouche's theorem shows that the zeros of the polynomial $P_{1}(z)- \lambda F_{2}(z)$ of degree $n-m \geq 1$ lie in $|z|<1$. Hence the polynomial 
 \[f(z) = F_{1}(z)\left(P_{1}(z) - \lambda F_{2}(z)\right)=P(z) - \lambda F(z)\]
 has all its zeros in $|z|\leq 1$ with at least one zero in $|z| <1$, so that we can write
 \[f(z)= (z-te^{i\delta})H(z)\]
 where $t <1$ and $H(z)$ is a polynomial of degree $n-1$ having all its zeros in $|z|\leq 1$. Applying Lemma \ref{l1} to the polynomial $f(z)$ with $k = 1$, we obtain for every $R >r \geq 1$ and $0 \leq \theta <2\pi$,
\begin{align*}
|f(Re^{i\theta})| =&|Re^{i\theta}-te^{i\delta}||H(Re^{i\theta})|\\
 \geq& |Re^{i\theta}-te^{i\delta}|\left(\frac{R+1}{r+1}\right)^{n-1}|H(re^{i\theta})|\\
=& \left(\frac{R+1}{r+1}\right)^{n-1}\frac{|Re^{i\theta}-te^{i\delta}|}{|re^{i\theta}-te^{i\delta}|}|(re^{i\theta}-te^{i\delta})H(re^{i\theta})|\\
\geq & \left(\frac{R+1}{r+1}\right)^{n-1}\left(\frac{R+t}{r+t}\right)|f(re^{i\theta})|.
\end{align*}
This implies  for $R > r \geq 1$ and $0 \leq \theta <2\pi$,
\begin{equation}\label{l3e2}
\left(\frac{r+t}{R+t}\right)|f(Re^{i\theta})|\geq \left(\frac{R+1}{r+1}\right)^{n-1}|f(re^{i\theta})|.
\end{equation}
Since $R>r \geq 1 > t$ so that $f(Re^{i\theta})\neq 0$ for $0 \leq \theta <2\pi$ and $\frac{1+r}{1+R}>\frac{r+t}{R+t}$, from inequality \eqref{l3e2}, we obtain $R>r\geq 1$ and $0 \leq \theta <2\pi$,
\begin{equation}\label{l3e3}
|f(Re^{i\theta})|> \left(\frac{R+1}{r+1}\right)^{n}|f(re^{i\theta})|.
\end{equation}
Equivalently,
\[|f(Rz)|> \left(\frac{R+1}{r+1}\right)^{n}|f(rz)|\]
for $|z|=1$ and $R>r\geq 1$. Hence for every $\alpha\in\mathbb{C}$ with $|\alpha|\leq 1$ and $R > r\geq 1,$ we have
\begin{align}\nonumber
|f(Rz)-\alpha f(rz)|&\geq  |f(Rz)|-|\alpha||f(rz)|\\
&> \left\{\left(\frac{R+1}{r+1}\right)^{n}-|\alpha|\right\}|f(rz)|, \,\,\,\,\,|z|=1.
\end{align}
Also, inequality \eqref{l3e3} can be written in the form
\begin{equation}\label{l3e4}
|f(re^{i\theta})|<\left(\frac{r+1}{R+1}\right)^{n}|f(Re^{i\theta})|
\end{equation}
for every $R>r\geq 1$ and $0 \leq \theta <2\pi.$ Since $f(Re^{i\theta}) \neq 0$ and $\left(\frac{r+1}{R+1}\right)^{n}<1$, from inequality \eqref{l3e4}, we obtain for $0 \leq \theta <2\pi$ and $R >r \geq 1$,
\[|f(re^{i\theta})|<|f(Re^{i\theta})|.\]
Equivalently,
\[|f(rz)|<|f(Rz)|\,\,\, \textrm{for}\,\,\,\, |z|=1.\]
Since all the zeros of $f(Rz)$ lie in $|z|\leq (1/R)<1$, a direct application of Rouche's theorem shows that the polynomial $f(Rz)-\alpha f(rz)$ has all its zeros in $|z|<1$ for every $\alpha\in\mathbb{C}$ with $|\alpha|\leq 1$. Applying Rouche's theorem again, it follows from \eqref{l3e3} that for $\alpha,\beta\in\mathbb{C}$ with $|\alpha|\leq 1,|\beta|\leq 1$ and $R >r \geq 1$, all the zeros of the polynomial
\begin{align*}
 T(z)=&f(Rz)-\alpha f(rz)+\beta\left\{\left(\frac{R+1}{r+1}\right)^{n}-|\alpha|\right\}f(rz)\\
 &=f(Rz)+\phi_{n}\left(R, r, \alpha, \beta\right)f(rz)\\
 &=\big(P(Rz)-\lambda F(Rz)\big)+\phi_{n}\left(R, r, \alpha, \beta\right)\big(P(rz)-\lambda F(rz)\big)\\
 &=\big(P(Rz)+\phi_{n}\left(R, r, \alpha, \beta\right)P(rz)\big)-\lambda \big(F(Rz)+\phi_{n}\left(R, r, \alpha, \beta\right)F(rz)\big)
\end{align*}
lie in $|z|<1$ for every $\lambda\in\mathbb{C}$ with $|\lambda| > 1$. Using Lemma \ref{l2} and the fact that $B$ is a linear operator, we conclude that all the zeros of polynomial
\begin{align*}
 W(z)&=B[T](z)\\
 & =(B[P\circ\sigma](z)+\phi_{n}\left(R, r, \alpha, \beta\right)B[P\circ\rho](z))\\
 &\qquad\qquad\qquad\qquad\qquad-\lambda (B[F\circ\sigma](z)+\phi_{n}\left(R, r, \alpha, \beta\right)B[F\circ\rho](z))
\end{align*}
also lie in $|z|<1$ for every $\lambda$ with $|\lambda|> 1$. This implies
\begin{equation}\label{l3e5}
|B[P\circ\sigma](z)+\phi_{n}\left(R, r, \alpha, \beta\right)B[P\circ\rho](z)|\leq |B[F\circ\sigma](z)+\phi_{n}\left(R, r, \alpha, \beta\right)B[F\circ\rho](z)| 
\end{equation}
for $|z|\geq 1$ and $R > r\geq 1$. If inequality \eqref{l3e5} is not true, then exist a point $z=z_0$ with $|z_0|\geq 1$ such that
\begin{equation*}
|B[P\circ\sigma](z_0)+\phi_{n}\left(R, r, \alpha, \beta\right)B[P\circ\rho](z_0)|> |B[F\circ\sigma](z_0)+\phi_{n}\left(R, r, \alpha, \beta\right)B[F\circ\rho](z_0)|.
\end{equation*}
But all the zeros of $F(Rz)$ lie in $|z|< 1$, therefore, it follows (as in case of $f(z)$) that all the zeros of $F(Rz)+\phi_{n}\left(R, r, \alpha, \beta\right)F(rz)$ lie in $|z|<1$. Hence by Lemma \ref{l2}, all the zeros of $B[F\circ\sigma](z)+ \phi_{n}\left(R, r, \alpha, \beta\right)B[F\circ\rho](z)$
also lie in $|z|<1$, which shows that
$$B[F\circ\sigma](z_0)+\phi_{n}\left(R, r, \alpha, \beta\right)B[F\circ\rho](z_0)\neq 0.$$ 
We take
\[\lambda = \frac{B[P\circ\sigma](z_0)+\phi_{n}\left(R, r, \alpha, \beta\right)B[P\circ\rho](z_0)}{B[F\circ\sigma](z_0)+\phi_{n}\left(R, r, \alpha, \beta\right)B[F\circ\rho](z_0)},\]
then $\lambda$ is a well defined real or complex number with $|\lambda|>1$ and with this choice of $\lambda$, we obtain $W(z_0)=0$. This contradicts the fact that all the zeros of $W(z)$ lie in $|z|<1$. Thus \eqref{l3e5} holds and
this completes the proof of Lemma \ref{l3}.
\end{proof}
\begin{lemma}\label{l2'}
If $P\in\mathscr P_n$ and $P(z)$ has all its zeros in $|z|\leq 1,$ then for every $\alpha,\beta\in\mathbb{C}$ with $|\alpha|\leq 1,|\beta|\leq 1$  and $|z|\geq 1,$
\begin{align}\label{le2'}
\left|B[P\circ\sigma](z)+\phi_{n}\left(R, r, \alpha, \beta\right) B[P\circ\rho](z)\right|\geq |R^n+\phi_{n}\left(R, r, \alpha, \beta\right)r^n||\Lambda_n||z|^nm
\end{align}
where $m={\min}_{|z|=1}|P(z)|,$ $B\in\mathcal{B}_n,$ $\sigma(z)=Rz,\,\rho(z)=rz,$  $\Lambda_n$ and
$\phi_{n}\left(R, r, \alpha, \beta\right)$ are defined by \eqref{12} and \eqref{phi} respectively.
\end{lemma}
\begin{proof}
By hypothesis, all the zeros of $P(z)$ lie in $|z|\leq 1$ and 
$$  m |z|^n\leq |P(z)|\,\,\,\,\textnormal{for}\,\,\,\,|z|=1. $$ 
We first show that the polynomial $g(z)=P(z)-\lambda mz^n$ has all its zeros in $|z|\leq 1$ for every $\lambda\in\mathbb{C}$ with $|\lambda|<1.$ This is obvious if $m=0,$ that is if $P(z)$ has a zero on $|z|=1.$ Henceforth, we assume $P(z)$ has all its zeros in $|z|<1,$ then $m>0$ and it follows by Rouche's theorem that the polynomial $g(z) $ has all its zeros in $|z|<1$ for every $\lambda\in\mathbb{C}$ with $|\lambda|<1.$ Proceeding similarly as in the proof of Lemma \ref{l3}, we obtain that for $\alpha,\beta\in\mathbb{C}$ with $|\alpha|\leq 1,|\beta|\leq 1$ and $R >r \geq 1$, all the zeros of the polynomial
\begin{align*}
 H(z)=&g(Rz)-\alpha g(rz)+\beta\left\{\left(\frac{R+1}{r+1}\right)^{n}-|\alpha|\right\}g(rz)\\
 &=g(Rz)+\phi_{n}\left(R, r, \alpha, \beta\right)g(rz)\\
 &=\big(P(Rz)-\lambda R^nz^nm\big)+\phi_{n}\left(R, r, \alpha, \beta\right)\big(P(rz)-\lambda r^nz^nm\big)\\
 &=\big(P(Rz)+\phi_{n}\left(R, r, \alpha, \beta\right)P(rz)\big)-\lambda \big(R^n+\phi_{n}\left(R, r, \alpha, \beta\right)r^n\big)mz^n
\end{align*}
lie in $|z|<1.$ Applying Lemma \ref{l1} to $H(z)$ and noting that $B$ is a linear operator, it follows that all the zeros of polynomial 
\begin{align}\label{p1'}\nonumber
B[H](z)=&\left\{B[P\circ\sigma](z)+\phi_{n}\left(R, r, \alpha, \beta\right) B[P\circ\rho](z)    \right\}\\&\qquad\qquad\qquad\qquad-\lambda\big(R^n+\phi_{n}\left(R, r, \alpha, \beta\right)r^n\big)m B[z^n]
\end{align} 
lie in $|z|<1.$ This gives
\begin{align}\label{p2}\nonumber
&\left|B[P\circ\sigma](z)+\phi_{n}\left(R, r, \alpha, \beta\right) B[P\circ\rho](z)\right|\\&\qquad\qquad\qquad\qquad\geq |R^n+\phi_{n}\left(R, r, \alpha, \beta\right)r^n||\Lambda_n||z|^nm\,\,\,\,\,\textnormal{for}\,\,\,\,\,|z|\geq 1.
\end{align}
If \eqref{p2} is not true, then there is point $w$ with $|w|\geq 1$ such that 
\begin{align}
\left|B[P\circ\sigma](w)+\phi_{n}\left(R, r, \alpha, \beta\right) B[P\circ\rho](w)\right|< |R^n+\phi_{n}\left(R, r, \alpha, \beta\right)r^n||\Lambda_n||w|^nm.
\end{align}
We choose
$$ \lambda=\dfrac{B[P\circ\sigma](w)+\phi_{n}\left(R, r, \alpha, \beta\right) B[P\circ\rho](w)}{R^n+\phi_{n}\left(R, r, \alpha, \beta\right)r^n||\Lambda_n||w|^nm.},   $$
then clearly $|\lambda|<1$ and with this choice of $\lambda,$ from \eqref{p1'}, we get $B[H](w)=0$ with $|w|\geq 1.$ This is clearly a contradiction to the fact that all the zeros of $H(z)$ lie in $|z|<1.$ Thus for every  $\alpha,\beta\in\mathbb{C}$ with $|\alpha|\leq 1, $ $|\beta|\leq 1,$
$$\left|B[P\circ\sigma](z)+\phi_{n}\left(R, r, \alpha, \beta\right) B[P\circ\rho](z)\right|\geq |R^n+\phi_{n}\left(R, r, \alpha, \beta\right)r^n||\Lambda_n||z|^nm$$
for $|z|\geq 1$ and $R>r\geq 1.$
\end{proof}
\begin{lemma}\label{l4}
If $P\in \mathscr{P}_n$ and $P(z)$ does not vanish in $\left|z\right|< 1$, then for every $\alpha,\beta\in\mathbb{C}
$ with $|\alpha| \leq 1, |\beta|\leq 1, R> r\geq 1$ and $|z|\geq 1,$
\begin{align}\nonumber
|B[P\circ\sigma](z) +& \phi_{n}\left(R, r, \alpha, \beta\right) B[P\circ\rho](z)|\\
&\leq |B[P^{*}\circ\sigma](z) + \phi_{n}\left(R, r, \alpha, \beta\right) B[P^{*}\circ\rho](z)|
\end{align}
where $P^{*}(z):=z^{n}\overline{P(1/\overline{z})}$, $B\in\mathcal B_n$, $\sigma(z):=Rz$, $\rho(z):=rz$, and $\phi_n\left(R,r, \alpha, \beta\right)$ is defined by \eqref{phi}.
\end{lemma}
\begin{proof}
 By hypothesis the polynomial $P(z)$ of degree $n$ does not vanish in $|z|< 1$, therefore, all the zeros of the polynomial $P^{*}(z)=z^{n}\overline{P(1/\overline{z})}$ of degree $n$ lie in $|z|\leq 1$. Applying Lemma \ref{l3} with $F(z)$ replaced by $P^{*}(z)$, it follows that 
\begin{align*}
&\left|B[P\circ\sigma](z) + \phi_{n}\left(R, r, \alpha, \beta\right) B[P\circ\rho](z)\right|\\&\qquad\qquad\qquad\qquad\qquad
\leq \left|B[P^{*}\circ\sigma](z) + \phi_{n}\left(R, r, \alpha, \beta\right) B[P^{*}\circ\rho](z)\right|
\end{align*}
for $|z|\geq 1, |\alpha|\leq 1, |\beta|\leq 1$ and $R> r\geq 1$. This proves the Lemma \ref{l4}.
\end{proof}
  \begin{lemma}\label{l3'}
      If $ P\in\mathscr{P}_n $ and $ P(z) $ has no zero in $\left|z\right|<1,$ then for every $ \alpha\in\mathbb{C} $ with $ |\alpha|\leq 1,$ $ R>r\geq 1 $ and $ |z|\geq 1 $,
     \begin{align}\label{le3'}\nonumber
     &\left|B[P\circ\sigma](z)+\phi_{n}\left(R, r, \alpha, \beta\right) B[P\circ\rho](z)\right|\\\nonumber&\qquad\leq \left|B[P^{\star}\circ\sigma](z)+\phi_{n}\left(R, r, \alpha, \beta\right) B[P^{\star}\circ\rho](z)\right|\\&\qquad\quad-\Big(| R^n+\phi_{n}\left(R, r, \alpha, \beta\right) r^n||\Lambda_n|-|1+\phi_{n}\left(R, r, \alpha, \beta\right)||\lambda_0|\Big)  m,
     \end{align}
       where $P^{\star}(z)=z^{n}\overline{P(1/\overline{z})} ,$ $m={\min}_{|z|=1}|P(z)|,$ $B\in\mathcal B_n,$ $\sigma(z)=Rz,$ $\rho(z)=rz,$ $\Lambda_n$ and $\phi_{n}\left(R, r, \alpha, \beta\right)$ are given by \eqref{12} and \eqref{phi} respectively.
       \end{lemma}
       \begin{proof}
By hypothesis $P(z)$ has all its zeros in $|z|\geq 1$ and 
\begin{align}\label{q1}
m\leq |P(z)|\,\,\,\textnormal{for}\,\,\,\,|z|=1.
\end{align}
We show $F(z)=P(z)+\lambda m$ does not vanish in $|z|<1$ for every $\lambda\in\mathbb{C}$ with $|\lambda|<1.$ This is obvious if $m=0$ that is, if $P(z)$ has a zero on $|z|=1.$ So we assume all the zeros of $P(z)$ lie in $|z|>1,$ then $m>0$ and by the maximum modulus principle, it follows from \eqref{q1},
\begin{align}\label{q2}
m< |P(z)|\,\,\,\textnormal{for}\,\,\,|z|<1.
\end{align} 
Now if $F(z)=P(z)+\lambda m=0$ for some $z_0$ with $|z_0|<1,$ then
\begin{align*}
P(z_0)+\lambda m=0
\end{align*} 
This implies
\begin{align}
|P(z_0)|= |\lambda|m\leq m,\,\,\,\textnormal{for}\,\,\,|z_0|<1
\end{align} 
which is clearly contradiction to \eqref{q2}. Thus the polynomial $F(z)$ does not vanish in $|z|<1$  for every $\lambda$ with $|\lambda|<1.$ Applying Lemma \ref{l3} to the polynomial $F(z),$ we get
\begin{align*}
|B[F\circ\sigma](z)+&\phi_{n}\left(R, r, \alpha, \beta\right) B[F\circ\rho](z)|\\&\leq |B[F^{\star}\circ\sigma](z)+\phi_{n}\left(R, r, \alpha, \beta\right) B[F^{\star}\circ\rho](z)|
\end{align*} 
for $|z|=1$ and $R>r\geq 1.$ Replacing $F(z)$ by $P(z)+\lambda m,$ we obtain
\begin{align}\label{q3}\nonumber
|B[P\circ\sigma](z)+\phi_{n}\left(R, r, \alpha, \beta\right)& B[P\circ\rho](z)+\lambda(1+\phi_{n}\left(R, r, \alpha, \beta\right))\lambda_0 m|\\\nonumber&\leq |B[P^{\star}\circ\sigma](z)+\phi_{n}\left(R, r, \alpha, \beta\right) B[P^{\star}\circ\rho](z)\\&\qquad\qquad\qquad+\bar{\lambda} (R^n+\phi_{n}\left(R, r, \alpha, \beta\right) r^n)\Lambda_n z^n m|
\end{align}
 Now choosing the argument of $\lambda$ in the right hand side of \eqref{q3} such that
 \begin{align*}\nonumber
|B[P^{\star}\circ\sigma](z)+\phi_{n}&\left(R, r, \alpha, \beta\right) B[P^{\star}\circ\rho](z)+\bar{\lambda} (R^n+\phi_{n}\left(R, r, \alpha, \beta\right) r^n)\Lambda_n z^n m|\\&=|B[P^{\star}\circ\sigma](z)+\phi_{n}\left(R, r, \alpha, \beta\right) B[P^{\star}\circ\rho](z)|\\&\qquad\qquad\qquad-|\bar{\lambda}|| R^n+\phi_{n}\left(R, r, \alpha, \beta\right) r^n||\Lambda_n| |z|^n m
\end{align*}
for $|z|=1,$which is possible by Lemma \ref{l2'},we get
\begin{align*}\nonumber
|B[P\circ\sigma](z)&+\phi_{n}\left(R, r, \alpha, \beta\right) B[P\circ\rho](z)|-|\lambda||1+\phi_{n}\left(R, r, \alpha, \beta\right)||\lambda_0| m\\\leq& |B[P^{\star}\circ\sigma](z)+\phi_{n}\left(R, r, \alpha, \beta\right) B[P^{\star}\circ\rho](z)|\\&\qquad\qquad\qquad-|\lambda|| R^n+\phi_{n}\left(R, r, \alpha, \beta\right) r^n||\Lambda_n| |z|^n m
\end{align*}
Equivalently,
\begin{align}\nonumber\label{cde}
|B[P\circ\sigma](z)&+\phi_{n}\left(R, r, \alpha, \beta\right) B[P\circ\rho](z)|\\\nonumber\leq& |B[P^{\star}\circ\sigma](z)+\phi_{n}\left(R, r, \alpha, \beta\right) B[P^{\star}\circ\rho](z)|\\&\quad-|\lambda|\Big(| R^n+\phi_{n}\left(R, r, \alpha, \beta\right) r^n||\Lambda_n|-|1+\phi_{n}\left(R, r, \alpha, \beta\right)||\lambda_0|\Big)  m.
\end{align}
Letting $|\lambda|\rightarrow 1$ in \eqref{cde} we obtain inequality \eqref{le3'} and this completes the proof of Lemma \ref{l3'}.
       \end{proof}
\indent Next we describe a result of Arestov \cite{2}.

For $\gamma = \left(\gamma_{0},\gamma_{1},\cdots,\gamma_{n}\right)\in \mathbb C^{n+1}\,\,\,\,\mbox{}\,\,\,\text{and}\,\,\,P(z)=\sum_{j=0}^{n}a_{j}z^{j}$, we define
$$C_{\gamma}P(z)=\sum_{j=0}^{n}\gamma_{j} a_{j}z^{j}.$$
\indent The operator $C_{\gamma}$ is said to be admissible if it preserves one of the following properties:
\begin{enumerate}
\item[(i)]   $P(z)$ has all its zeros in $\left\{z\in \mathbb C: |z| \leq 1\right\}$,
\item[(ii)]  $P(z)$ has all its zeros in $\left\{z\in \mathbb C: |z| \geq 1\right\}$. 
\end{enumerate}
The result of Arestov may now be stated as follows.
\begin{lemma}\cite[Theorem 2]{2}\label{l5}
\textit{ Let} $\phi(x)=\psi(\log x)$ \textit{where} $\psi$ \textit{is a convex non-decreasing function on} $\mathbb R$. \textit{Then for all} $P\in \mathscr{P}_n$ \textit{and each admissible operator} $\Lambda_{\gamma}$,
\begin{equation}
\int_{0}^{2\pi}\phi\left(|C_{\gamma}P(e^{i\theta})|\right)d\theta \leq \int_{0}^{2\pi}\phi\left(c(\gamma,n)|P(e^{i\theta})|\right)d\theta\nonumber
\end{equation}
where $c(\gamma,n)= \max \left(|\gamma_{0}|,|\gamma_{n}|\right)$.
\end{lemma}
In particular Lemma \ref{l5} applies with $\phi : x\rightarrow x^{p}$ for every $ p \in(0,\infty)$ and $\phi : x\rightarrow \log x$ as well. Therefore, we have for $0 \leq p <\infty$,
\begin{equation}\label{av2}
 \left\{\int_{0}^{2\pi}\phi\left(|C_{\gamma}P(e^{i\theta})|^{p}\right)d\theta\right\}^{1/p}\leq c(\gamma,n)\left\{\int_{0}^{2\pi}\left|P(e^{i\theta})\right|^{p}d\theta\right\}^{1/p}.
\end{equation}
From Lemma \ref{l5}, we deduce the following result.
 \begin{lemma}\label{l6}
 \textit{ If} $P\in \mathscr{P}_n$ \textit{and} $P(z)$ \textit{does not vanish in} $\left|z\right|< 1$, \textit {then for each} $p > 0$, $R > 1$ and $\eta$ real, $0 \leq\eta <2\pi$,
 \begin{align}\nonumber
 \int_{0}^{2\pi}|\big(B[P\circ\sigma](e^{i\theta})&+\phi_n(R,r,\alpha,\beta)B[P\circ\rho](e^{i\theta})\big)e^{i\eta}\\\nonumber
 &+\big(B[P^{*}\circ\sigma]^{*}(e^{i\theta})+\phi_n(R,r,\bar{\alpha},\bar{\beta})B[P^{*}\circ\rho]^{*}(e^{i\theta})\big)|^{p}d\theta\\\nonumber
 \leq|(R^{n}+\phi_{n}(R,r,&\alpha,\beta)r^{n})\Lambda_{n}e^{i\eta}+(1+\phi_{n}(R,r,\bar{\alpha},\bar{\beta}))\bar{\lambda_{0}}|^{p}\int_{0}^{2\pi}\left|P(e^{i\theta})\right|^{p}d\theta\nonumber
 \end{align}
 where $B\in\mathcal B_n$, $\sigma(z):=Rz$, $\rho(z):=rz$, $B[P^{*}\circ\sigma]^{*}(z):=(B[P^{*}\circ\sigma](z))^{*}$, $\Lambda_n$ \textit{and}
 $\phi_{n}\left(R, r, \alpha, \beta\right)$ \textit{are defined by} \eqref{12} \textit{and} \eqref{phi} \textit{respectively}.
 \end{lemma}
\begin{proof}
 Since $P(z)$ does not vanish in $\left|z\right|< 1$ and $P^{*}(z)= z^{n}\overline{P(1/\bar{z})}$, by Lemma \ref{l4}, we have for $R>r\geq1$,
\begin{align}\nonumber\label{l6e1}
|B[P\circ\sigma](z) + \phi_{n}&\left(R, r, \alpha, \beta\right)B[P\circ\rho](z)|\\ 
&\leq |B[P^{*}\circ\sigma](z) + \phi_{n}\left(R, r, \alpha, \beta\right) B[P^{*}\circ\rho](z)|
\end{align}
Also, since\\
$P^{*}(Rz)+\phi_{n}\left(R, r, \alpha, \beta\right)P^{*}(rz)= R^{n}z^{n}\overline{P(1/R\bar{z})}+\phi_{n}\left(R, r, \alpha, \beta\right)
r^{n}z^{n}\overline{P(1/r\bar{z})}$,
therefore,  
\begin{align}\nonumber
&B[P^{*}\circ\sigma](z)+\phi_{n}(R,r,\alpha,\beta)B[P^{*}\circ\rho](z)\\\nonumber
&=\lambda_{0}\big(R^{n}z^{n}\overline{P(1/R\bar{z})}+\phi_{n}\left(R, r, \alpha, \beta\right)r^{n}z^{n}\overline{P(1/r\bar{z})}\big)
+\lambda_{1}\left(\frac{nz}{2}\right)\Big(nR^{n}z^{n-1}\overline{P(1/R\bar{z})}\\\nonumber
&\quad -R^{n-1}z^{n-2}\overline{P^{\prime}(1/R\bar{z})}+\phi_{n}\left(R, r, \alpha, \beta\right)\big(nr^{n}z^{n-1}\overline{P(1/r\bar{z})}
-r^{n-1}z^{n-2}\overline{P^{\prime}(1/r\bar{z})}\big)\Big)\\\nonumber
&\quad + \frac{\lambda_{2}}{2!}\left(\frac{nz}{2}\right)^{2}\Big(n(n-1)R^{n}z^{n-2}\overline{P(1/R\bar{z})}
-2(n-1)R^{n-1}z^{n-3}\overline{P^{\prime}(1/R\bar{z})}\\\nonumber
&\quad + R^{n-2}z^{n-4}\overline{P^{\prime\prime}(1/R\bar{z})}
+\phi_{n}\left(R, r, \alpha, \beta\right)\big(n(n-1)r^{n}z^{n-2}\overline{P(1/r\bar{z})}\\\nonumber
&\quad -2(n-1)r^{n-1}z^{n-3}\overline{P^{\prime}(1/r\bar{z})}
+ r^{n-2}z^{n-4}\overline{P^{\prime\prime}(1/r\bar{z})}\big)\Big)
\end{align}
and hence,
\begin{align}\nonumber
 B[P^{*}\circ\sigma]^{*}&(z)+\phi\left(R, r, \bar{\alpha}, \bar{\beta}\right)B[P^{*}\circ\rho]^{*}(z)\\\nonumber
=&\big(B[P^{*}\circ\sigma](z)+\phi_{n}\left(R, r, \alpha, \beta\right)B[P^{*}\circ\rho](z)\big)^{*}\\\nonumber 
=&\left(\bar{\lambda_{0}}+\bar{\lambda_{1}}\frac{n^{2}}{2}+\bar{\lambda_{2}}\frac{n^{3}(n-1)}{8}\right)\left(R^{n}P(z/R)+\phi\left(R, r, \bar{\alpha}, \bar{\beta}\right)r^{n}P(z/r)\right)\\\nonumber
&-\left(\bar{\lambda_{1}}\frac{n}{2}+\bar{\lambda_{2}}\frac{n^{2}(n-1)}{4}\right)\Big(R^{n-1}zP^{\prime}(z/R)+\phi\left(R, r, \bar{\alpha}, \bar{\beta}\right)r^{n-1}zP^{\prime}(z/r)\Big)\\
&+\bar{\lambda_{2}}\frac{n^{2}}{8}\Big(R^{n-2}z^{2}P^{\prime\prime}(z/R)+\phi\left(R, r, \bar{\alpha}, \bar{\beta}\right)r^{n-2}z^{2}P^{\prime\prime}(z/r)\Big).
\end{align}
Also, for $|z|=1$
\begin{align}\nonumber
|B[P^{*}\circ\sigma](z)+&\phi_{n}\left(R, r, \alpha, \beta\right)B[P^{*}\circ\rho](z)|\\\nonumber
&=|B[P^{*}\circ\sigma]^{*}(z)+\phi\left(R, r, \bar{\alpha}, \bar{\beta}\right)B[P^{*}\circ\rho]^{*}(z)|.
\end{align}
Using this in \eqref{l6e1}, we get for $|z|=1$ and $R >r\geq1$,
\begin{align}\nonumber 
|B[P\circ\sigma](z) + &\phi_{n}\left(R, r, \alpha, \beta\right)B[P\circ\rho](z)|\\\nonumber
&\leq|B[P^{*}\circ\sigma]^{*}(z)+\phi\left(R, r, \bar{\alpha}, \bar{\beta}\right)B[P^{*}\circ\rho]^{*}(z)|.
\end{align}
\indent Since  all the zeros of $P^{*}(z)$ lie in $ |z|\leq1$, as before, all the zeros of $P^{*}(Rz)+\phi_{n}(R,r,\alpha,\beta)P^{*}(rz)$ lie in $|z|<1$ for all real or complex numbers $\alpha, \beta$ with $|\alpha|\leq1$, $|\beta|\leq1$ and $R > r\geq 1$. Hence by Lemma \ref{l2}, all the zeros of $B[P^{*}\circ\sigma](z)+\phi_{n}(R,r,\alpha,\beta)B[P^{*}\circ\rho](z)$ lie in $|z| < 1$, therefore, all the zeros of $B[P^{*}\circ\sigma]^{*}(z)+\phi_{n}(R,r,\bar{\alpha},\bar{\beta})B[P^{*}\circ\rho]^{*}(z)$ lie in $|z| > 1$. Hence by the maximum modulus principle, 
\begin{align}\nonumber\label{l6e2}
|B[P\circ\sigma](z) + &\phi_{n}\left(R, r, \alpha, \beta\right) B[P^{*}\circ\rho](z)|\\
&<|B[P^{*}\circ\sigma]^{*}(z)+\phi\left(R, r, \bar{\alpha}, \bar{\beta}\right)B[P^{*}\circ\rho]^{*}(z)| \quad \textrm{for}\quad |z|<1.
\end{align}
A direct application of Rouche's theorem shows that
\begin{align}\nonumber
C_{\gamma}P(z)=&\big(B[P\circ\sigma](z)+\phi_{n}(R, r, \alpha, \beta) B[P\circ\rho](z)\big)e^{i\eta}\\\nonumber
&+\big(B[P^{*}\circ\sigma]^{*}(z)+\phi_{n}(R, r, \bar{\alpha}, \bar{\beta})B[P^{*}\circ\rho]^{*}(z)\big)\\\nonumber
=&\left\{(R^{n}+\phi_{n}(R,r,\alpha,\beta)r^{n})\Lambda_{n} e^{i\eta}
+(1+\phi_{n}(R,r,\bar{\alpha},\bar{\beta}))\bar{\lambda_{0}}\right\}a_{n}z^{n}\\\nonumber
&+\cdots+\left\{(R^{n}+\phi_{n}(R,r,\bar{\alpha},\bar{\beta})r^{n})\bar{\Lambda_{n}}
+e^{i\eta}(1+\phi_{n}(R,r, \alpha,\beta))\lambda_{0}\right\}a_{0}\nonumber
\end{align}
does not vanish in $|z| < 1$. Therefore, $C_{\gamma}$ is an admissible operator. Applying \eqref{av2} of Lemma \ref{l5}, the desired result follows immediately for each $p > 0$.
\end{proof}
 We also need the following lemma \cite{ar97}.
 \begin{lemma}\label{abc}
 If $A,B,C$ are non-negative real numbers such that $B+C\leq A,$ then for each real number $\gamma,$
 $$ |(A-C)e^{i\gamma}+(B+C)|\leq |Ae^{i\gamma}+B|.    $$
 \end{lemma}
\section{Proof of the Theorems} 
\begin{proof}[\textnormal{\bf Proof of Theorem \ref{t2}}]
   By hypothesis $P(z)$ does not vanish in $|z|<1,$ therefore by Lemma \ref{l3'}, we have
   \begin{align}\label{r1}\nonumber
     &\left|B[P\circ\sigma](z)+\phi_{n}\left(R, r, \alpha, \beta\right) B[P\circ\rho](z)\right|\\\nonumber&\qquad\leq \left|B[P^{\star}\circ\sigma](z)+\phi_{n}\left(R, r, \alpha, \beta\right) B[P^{\star}\circ\rho](z)\right|\\&\qquad\quad-\Big(| R^n+\phi_{n}\left(R, r, \alpha, \beta\right) r^n||\Lambda_n|-|1+\phi_{n}\left(R, r, \alpha, \beta\right)||\lambda_0|\Big)  m,
   \end{align}
    for $|z|=1,$ $|\alpha|\leq 1$ and $R>r\geq 1$ where $P^{\star}(z)=z^n\overline{P(1/\overline{z})}.$
  Since $B[P^{\star}\circ\sigma]^{\star}(z)+\phi_{n}\left(R, r,\bar{\alpha},\bar{\beta}\right)B[P^{\star}\circ\rho]^{\star}(z)$ is the conjugate of $B[P^{\star}\circ\sigma](z)+\phi_{n}\left(R, r, \alpha, \beta\right) B[P^{\star}\circ\rho](z)$ and
  \begin{align*}
 |B[P^{\star}\circ\sigma]^{\star}(z)&+\phi_{n}\left(R, r, \bar{\alpha}, \bar{\beta}\right)B[P^{\star}\circ\rho]^{\star}(z)|\\&=|B[P^{\star}\circ\sigma](z)+\phi_{n}\left(R, r, \alpha, \beta\right) B[P^{\star}\circ\rho](z)|
  \end{align*}
  Thus \eqref{r1} can be written as 
  \begin{align}\label{r2}\nonumber
\big|B&[P\circ\sigma](z)+\phi_{n}\left(R, r, \alpha, \beta\right) B[P\circ\rho](z)\big|\\\nonumber&\qquad\qquad+\dfrac{\Big(| R^n+\phi_{n}\left(R, r, \alpha, \beta\right) r^n||\Lambda_n|-|1+\phi_{n}\left(R, r, \alpha, \beta\right)||\lambda_0|\Big)m}{2}\\\nonumber&\leq\left|B[P^{\star}\circ\sigma]^{\star}(z)+\phi_{n}\left(R, r,\bar{\alpha},\bar{\beta}\right) B[P^{\star}\circ\rho]^{\star}(z)\right|\\&\qquad\qquad-\dfrac{\Big(| R^n+\phi_{n}\left(R, r, \alpha, \beta\right) r^n||\Lambda_n|-|1+\phi_{n}\left(R, r, \alpha, \beta\right)||\lambda_0|\Big)m}{2},
  \end{align}
for $|z|=1.$  Taking 
  $$A=\left|B[P^{\star}\circ\sigma]^{\star}(z)+\phi_{n}\left(R, r,\bar{\alpha},\bar{\beta}\right) B[P^{\star}\circ\rho]^{\star}(z)\right|  $$
  $$B=\big|B[P\circ\sigma](z)+\phi_{n}\left(R, r, \alpha, \beta\right) B[P\circ\rho](z)\big|,$$
  and 
  $$  C=\dfrac{\Big(| R^n+\phi_{n}\left(R, r, \alpha, \beta\right) r^n||\Lambda_n|-|1+\phi_{n}\left(R, r, \alpha, \beta\right)||\lambda_0|\Big)m}{2} $$
  in Lemma \ref{abc} and noting by \eqref{r2} that 
  $$ B+C\leq A-C\leq A,   $$
  we get for every real $\gamma$,
  \begin{align*}
  \Bigg|\Bigg\{\big|B[P^{\star}&\circ\sigma]^{\star}(e^{i\theta})+\phi_{n}\left(R, r,\bar{\alpha},\bar{\beta}\right) B[P^{\star}\circ\rho]^{\star}(e^{i\theta})\big|\\&\qquad-\dfrac{\Big(| R^n+\phi_{n}\left(R, r, \alpha, \beta\right) r^n||\Lambda_n|-|1+\phi_{n}\left(R, r, \alpha, \beta\right)||\lambda_0|\Big)m}{2}\Bigg\}e^{i\gamma}\\+\Bigg\{\big|B&[P\circ\sigma](e^{i\theta})+\phi_{n}\left(R, r, \alpha, \beta\right) B[P\circ\rho](e^{i\theta})\big|\\&\qquad+\dfrac{\Big(| R^n+\phi_{n}\left(R, r, \alpha, \beta\right) r^n||\Lambda_n|-|1+\phi_{n}\left(R, r, \alpha, \beta\right)||\lambda_0|\Big)m}{2}\Bigg|\\\leq &\Big|\big|B[P^{\star}\circ\sigma]^{\star}(e^{i\theta})+\phi_{n}\left(R, r,\bar{\alpha},\bar{\beta}\right) B[P^{\star}\circ\rho]^{\star}(e^{i\theta})\big|e^{i\gamma}\\&\qquad\qquad\qquad\qquad\quad+\big|B[P\circ\sigma](e^{i\theta})+\phi_{n}\left(R, r, \alpha, \beta\right) B[P\circ\rho](e^{i\theta})\big|\Big|.
  \end{align*}
  This implies for each $p>0,$
  \begin{align}\nonumber\label{r3}
 \int\limits_{0}^{2\pi} \Bigg|\Bigg\{\big|B[P^{\star}&\circ\sigma]^{\star}(e^{i\theta})+\phi_{n}\left(R, r,\bar{\alpha},\bar{\beta}\right) B[P^{\star}\circ\rho]^{\star}(e^{i\theta})\big|\\\nonumber&\qquad-\dfrac{\Big(| R^n+\phi_{n}\left(R, r, \alpha, \beta\right) r^n||\Lambda_n|-|1+\phi_{n}\left(R, r, \alpha, \beta\right)||\lambda_0|\Big)m}{2}\Bigg\}e^{i\gamma}\\\nonumber+\Bigg\{\big|B&[P\circ\sigma](e^{i\theta})+\phi_{n}\left(R, r, \alpha, \beta\right) B[P\circ\rho](e^{i\theta})\big|\\\nonumber&\qquad+\dfrac{\Big(| R^n+\phi_{n}\left(R, r, \alpha, \beta\right) r^n||\Lambda_n|-|1+\phi_{n}\left(R, r, \alpha, \beta\right)||\lambda_0|\Big)m}{2}\Bigg|^pd\theta\\\nonumber\leq\int\limits_{0}^{2\pi} &\Big|\big|B[P^{\star}\circ\sigma]^{\star}(e^{i\theta})+\phi_{n}\left(R, r,\bar{\alpha},\bar{\beta}\right) B[P^{\star}\circ\rho]^{\star}(e^{i\theta})\big|e^{i\gamma}\\&\quad\quad\qquad\qquad\quad+\big|B[P\circ\sigma](e^{i\theta})+\phi_{n}\left(R, r, \alpha, \beta\right) B[P\circ\rho](e^{i\theta})\big|\Big|^pd\theta.
  \end{align}
  Integrating both sides of \eqref{r3} with respect to $\gamma$ from $0$ to $2\pi,$ we get with the help of Lemma \ref{l6} for each $p>0,$
  \begin{align}\nonumber\label{r4}
 \int\limits_{0}^{2\pi} \int\limits_{0}^{2\pi} \Bigg|\Bigg\{\big|B[P^{\star}&\circ\sigma]^{\star}(e^{i\theta})+\phi_{n}\left(R, r,\bar{\alpha},\bar{\beta}\right) B[P^{\star}\circ\rho]^{\star}(e^{i\theta})\big|\\\nonumber&-\dfrac{\Big(| R^n+\phi_{n}\left(R, r, \alpha, \beta\right) r^n||\Lambda_n|-|1+\phi_{n}\left(R, r, \alpha, \beta\right)||\lambda_0|\Big)m}{2}\Bigg\}e^{i\gamma}\\\nonumber+\Bigg\{\big|B&[P\circ\sigma](e^{i\theta})+\phi_{n}\left(R, r, \alpha, \beta\right) B[P\circ\rho](e^{i\theta})\big|\\\nonumber&+\dfrac{\Big(| R^n+\phi_{n}\left(R, r, \alpha, \beta\right) r^n||\Lambda_n|-|1+\phi_{n}\left(R, r, \alpha, \beta\right)||\lambda_0|\Big)m}{2}\Bigg|^pd\theta d\gamma\\\nonumber\leq\int\limits_{0}^{2\pi}\int\limits_{0}^{2\pi} \Big|\big|B[&P^{\star}\circ\sigma]^{\star}(e^{i\theta})+\phi_{n}\left(R, r,\bar{\alpha},\bar{\beta}\right) B[P^{\star}\circ\rho]^{\star}(e^{i\theta})\big|e^{i\gamma}\\\nonumber&\quad\quad\qquad\qquad+\big|B[P\circ\sigma](e^{i\theta})+\phi_{n}\left(R, r, \alpha, \beta\right) B[P\circ\rho](e^{i\theta})\big|\Big|^pd\theta d\gamma\\\nonumber\leq\int\limits_{0}^{2\pi}\Bigg\{\int\limits_{0}^{2\pi} \Big|&\big|B[P^{\star}\circ\sigma]^{\star}(e^{i\theta})+\phi_{n}\left(R, r,\bar{\alpha},\bar{\beta}\right) B[P^{\star}\circ\rho]^{\star}(e^{i\theta})\big|e^{i\gamma}\\\nonumber&\quad\quad\qquad\qquad+\big|B[P\circ\sigma](e^{i\theta})+\phi_{n}\left(R, r, \alpha, \beta\right) B[P\circ\rho](e^{i\theta})\big|\Big|^pd\gamma\Bigg\}\theta \\\nonumber\leq\int\limits_{0}^{2\pi}\Bigg\{\int\limits_{0}^{2\pi} &\Bigg|\Big(B[P^{\star}\circ\sigma]^{\star}(e^{i\theta})+\phi_{n}\left(R, r,\bar{\alpha},\bar{\beta}\right) B[P^{\star}\circ\rho]^{\star}(e^{i\theta})\Big)e^{i\gamma}\\\nonumber&\quad\quad\qquad\qquad+\Big(B[P\circ\sigma](e^{i\theta})+\phi_{n}\left(R, r, \alpha, \beta\right) B[P\circ\rho](e^{i\theta})\Big)\Bigg|^pd\gamma\Bigg\}\theta \\\nonumber\leq\int\limits_{0}^{2\pi}\Bigg\{\int\limits_{0}^{2\pi} &\Bigg|\Big(B[P^{\star}\circ\sigma]^{\star}(e^{i\theta})+\phi_{n}\left(R, r,\bar{\alpha},\bar{\beta}\right) B[P^{\star}\circ\rho]^{\star}(e^{i\theta})\Big)e^{i\gamma}\\\nonumber&\quad\quad\qquad\qquad+\Big(B[P\circ\sigma](e^{i\theta})+\phi_{n}\left(R, r, \alpha, \beta\right) B[P\circ\rho](e^{i\theta})\Big)\Bigg|^pd\theta\Bigg\}\gamma\\\nonumber\leq \int\limits_{0}^{2\pi}&\Big|(R^{n}+\phi_{n}(R,r,\alpha,\beta)r^{n})\Lambda_{n}e^{i\gamma}+(1+\phi_{n}(R,r,\bar{\alpha},\bar{\beta}))\bar{\lambda_{0}}\Big|^{p}d\gamma\\&\qquad\qquad\qquad\qquad\times\int_{0}^{2\pi}\left|P(e^{i\theta})\right|^{p}d\theta
  \end{align}
  Now it can be easily verified that for every real number $\gamma$ and $s \geq 1$,
  $$\left|s+e^{i\alpha}\right| \geq \left|1+e^{i\alpha}\right|.$$
  This implies for each $ p > 0$,
  \begin{equation}\label{r5}
  \int_{0}^{2\pi}\left|s+e^{i\gamma}\right|^{p}d\gamma \geq \int_{0}^{2\pi}\left|1+e^{i\gamma}\right|^{p}d\gamma. 
  \end{equation}
  If
  \begin{align*}
  \big|B[P\circ\sigma]&(e^{i\theta})+\phi_{n}(R,r,\alpha,\beta) B[P\circ\rho](e^{i\theta})\big|\\+&\dfrac{\Big(| R^n+\phi_{n}\left(R, r, \alpha, \beta\right) r^n||\Lambda_n|-|1+\phi_{n}\left(R, r, \alpha, \beta\right)||\lambda_0|\Big)m}{2}\neq 0,
  \end{align*}
   we take $$ s = \dfrac{\splitfrac{\big|B[P^{\star}\circ\sigma]^{\star}(e^{i\theta})+\phi_{n}\left(R, r,\bar{\alpha},\bar{\beta}\right) B[P^{\star}\circ\rho]^{\star}(e^{i\theta})\big|}{+\dfrac{\Big(| R^n+\phi_{n}\left(R, r, \alpha, \beta\right) r^n||\Lambda_n|-|1+\phi_{n}\left(R, r, \alpha, \beta\right)||\lambda_0|\Big)m}{2}}}{\splitfrac{ \big|B[P\circ\sigma](e^{i\theta})+\phi_{n}(R,r,\alpha,\beta) B[P\circ\rho](e^{i\theta})\big|}{+\dfrac{\Big(| R^n+\phi_{n}\left(R, r, \alpha, \beta\right) r^n||\Lambda_n|-|1+\phi_{n}\left(R, r, \alpha, \beta\right)||\lambda_0|\Big)m}{2}}},$$ then by \eqref{r2}, $s \geq 1$ and we get with the help of \eqref{r5},
   \begin{align}\nonumber\label{r6}
  \int\limits_{0}^{2\pi} \Bigg|\Bigg\{&\big|B[P^{\star}\circ\sigma]^{\star}(e^{i\theta})+\phi_{n}\left(R, r,\bar{\alpha},\bar{\beta}\right)B[P^{\star}\circ\rho]^{\star}(e^{i\theta})\big|\\\nonumber&\quad\quad-\dfrac{\Big(| R^n+\phi_{n}\left(R, r, \alpha, \beta\right) r^n||\Lambda_n|-|1+\phi_{n}\left(R, r, \alpha, \beta\right)||\lambda_0|\Big)m}{2}\Bigg\}e^{i\gamma}\\\nonumber+\Bigg\{&\big|B[P\circ\sigma](e^{i\theta})+ \phi_{n}\left(R, r, \alpha, \beta\right)B[P\circ\rho](e^{i\theta})\big|\\\nonumber&\quad\quad+\dfrac{\Big(| R^n+\phi_{n}\left(R, r, \alpha, \beta\right) r^n||\Lambda_n|-|1+\phi_{n}\left(R, r, \alpha, \beta\right)||\lambda_0|\Big)m}{2}\Bigg\}\Bigg|^p d\gamma\\\nonumber=\Bigg|&\big|B[P\circ\sigma](e^{i\theta})+ \phi_{n}\left(R, r, \alpha, \beta\right)B[P\circ\rho](e^{i\theta})\big|\\\nonumber&\quad\quad+\dfrac{\Big(| R^n+\phi_{n}\left(R, r, \alpha, \beta\right) r^n||\Lambda_n|-|1+\phi_{n}\left(R, r, \alpha, \beta\right)||\lambda_0|\Big)m}{2}\Bigg|^p\\\nonumber\times\int\limits_{0}^{2\pi}&\left|e^{i\gamma}+\dfrac{\splitfrac{\big|B[P^{\star}\circ\sigma]^{\star}(e^{i\theta})+\phi_{n}\left(R, r,\bar{\alpha},\bar{\beta}\right) B[P^{\star}\circ\rho]^{\star}(e^{i\theta})\big|}{-\dfrac{\Big(| R^n+\phi_{n}\left(R, r, \alpha, \beta\right) r^n||\Lambda_n|-|1+\phi_{n}\left(R, r, \alpha, \beta\right)||\lambda_0|\Big)m}{2}}}{\splitfrac{ \big|B[P\circ\sigma](e^{i\theta})+\phi_{n}(R,r,\alpha,\beta) B[P\circ\rho](e^{i\theta})\big|}{+\dfrac{\Big(| R^n+\phi_{n}\left(R, r, \alpha, \beta\right) r^n||\Lambda_n|-|1+\phi_{n}\left(R, r, \alpha, \beta\right)||\lambda_0|\Big)m}{2}}}\right|^pd\gamma\\\nonumber=\Bigg|&\big|B[P\circ\sigma](e^{i\theta})+ \phi_{n}\left(R, r, \alpha, \beta\right)B[P\circ\rho](e^{i\theta})\big|\\\nonumber&\quad\quad+\dfrac{\Big(| R^n+\phi_{n}\left(R, r, \alpha, \beta\right) r^n||\Lambda_n|-|1+\phi_{n}\left(R, r, \alpha, \beta\right)||\lambda_0|\Big)m}{2}\Bigg|^p\\\nonumber\times\int\limits_{0}^{2\pi}&\left|e^{i\gamma}+\left|\dfrac{\splitfrac{\big|B[P^{\star}\circ\sigma]^{\star}(e^{i\theta})+\phi_{n}\left(R, r,\bar{\alpha},\bar{\beta}\right) B[P^{\star}\circ\rho]^{\star}(e^{i\theta})\big|}{-\dfrac{\Big(| R^n+\phi_{n}\left(R, r, \alpha, \beta\right) r^n||\Lambda_n|-|1+\phi_{n}\left(R, r, \alpha, \beta\right)||\lambda_0|\Big)m}{2}}}{\splitfrac{ \big|B[P\circ\sigma](e^{i\theta})+\phi_{n}(R,r,\alpha,\beta) B[P\circ\rho](e^{i\theta})\big|}{+\dfrac{\Big(| R^n+\phi_{n}\left(R, r, \alpha, \beta\right) r^n||\Lambda_n|-|1+\phi_{n}\left(R, r, \alpha, \beta\right)||\lambda_0|\Big)m}{2}}}\right|^p\,\right|d\gamma\\\nonumber\geq &\Bigg|\big|B[P\circ\sigma](e^{i\theta})+ \phi_{n}\left(R, r, \alpha, \beta\right)B[P\circ\rho](e^{i\theta})\big|\\+&\dfrac{\Big(| R^n+\phi_{n}\left(R, r, \alpha, \beta\right) r^n||\Lambda_n|-|1+\phi_{n}\left(R, r, \alpha, \beta\right)||\lambda_0|\Big)m}{2}\Bigg|^p\int\limits_{0}^{2\pi}|1+e^{i\gamma}|^pd\gamma.
   \end{align}
  
  For 
   \begin{align*}
    \big|B[P\circ\sigma]&(e^{i\theta})+\phi_{n}(R,r,\alpha,\beta) B[P\circ\rho](e^{i\theta})\big|\\+&\dfrac{\Big(| R^n+\phi_{n}\left(R, r, \alpha, \beta\right) r^n||\Lambda_n|-|1+\phi_{n}\left(R, r, \alpha, \beta\right)||\lambda_0|\Big)m}{2}\neq 0,
    \end{align*}
   then \eqref{r6}  is trivially true. Using this in \eqref{r4}, we conclude for every $\alpha,\beta\in\mathbb{C}$ with $|\alpha|\leq 1,|\beta|\leq 1$ $R>r\geq 1$ and $ p > 0$,
  \begin{align*}
  \int\limits_{0}^{2\pi}&\Bigg|\big|B[P\circ\sigma](e^{i\theta})+ \phi_{n}\left(R, r, \alpha, \beta\right)B[P\circ\rho](e^{i\theta})\big|\\&+\dfrac{\Big(| R^n+\phi_{n}\left(R, r, \alpha, \beta\right) r^n||\Lambda_n|-|1+\phi_{n}\left(R, r, \alpha, \beta\right)||\lambda_0|\Big)m}{2}\Bigg|^pd\theta\int\limits_{0}^{2\pi}|1+e^{i\gamma}|^pd\gamma\\ \nonumber
  &\leq \int\limits_{0}^{2\pi}\Big|(R^{n}+\phi_{n}(R,r,\alpha,\beta)r^{n})\Lambda_{n}e^{i\gamma}+(1+\phi_{n}(R,r,\bar{\alpha},\bar{\beta}))\bar{\lambda_{0}}\Big|^{p}d\gamma\int_{0}^{2\pi}\left|P(e^{i\theta})\right|^{p}d\theta
  \end{align*}
  This gives for every $\delta,\alpha,\beta$ with $|\delta|\leq 1,$ $|\alpha|\leq 1,$ $|\beta|\leq 1,$ $R>r\geq 1$ and $\gamma$ real
  \begin{align}\nonumber\label{r7}
  \int\limits_{0}^{2\pi}&\Bigg|B[P\circ\sigma](e^{i\theta})+ \phi_{n}\left(R, r, \alpha, \beta\right)B[P\circ\rho](e^{i\theta})\\\nonumber+\delta&\dfrac{\Big(| R^n+\phi_{n}\left(R, r, \alpha, \beta\right) r^n||\Lambda_n|-|1+\phi_{n}\left(R, r, \alpha, \beta\right)||\lambda_0|\Big)m}{2}\Bigg|^pd\theta\int\limits_{0}^{2\pi}|1+e^{i\gamma}|^pd\gamma\\
   &\leq \int\limits_{0}^{2\pi}\Big|(R^{n}+\phi_{n}(R,r,\alpha,\beta)r^{n})\Lambda_{n}e^{i\gamma}+(1+\phi_{n}(R,r,\bar{\alpha},\bar{\beta}))\bar{\lambda_{0}}\Big|^{p}d\gamma\int_{0}^{2\pi}\left|P(e^{i\theta})\right|^{p}d\theta
  \end{align} 
  Since
  \begin{align}\nonumber\label{r8}
\int\limits_{0}^{2\pi}&\Big|(R^{n}+\phi_{n}(R,r,\alpha,\beta)r^{n})\Lambda_{n}e^{i\gamma}+(1+\phi_{n}(R,r,\bar{\alpha},\bar{\beta}))\bar{\lambda_{0}}\Big|^{p}d\gamma\int_{0}^{2\pi}\left|P(e^{i\theta})\right|^{p}d\theta\\\nonumber
 &=\int\limits_{0}^{2\pi}\Big||(R^{n}+\phi_{n}(R,r,\alpha,\beta)r^{n})\Lambda_{n}|e^{i\gamma}+|(1+\phi_{n}(R,r,\bar{\alpha},\bar{\beta}))\bar{\lambda_{0}}|\Big|^{p}d\gamma\int_{0}^{2\pi}\left|P(e^{i\theta})\right|^{p}d\theta\\\nonumber
 &=\int\limits_{0}^{2\pi}\Big||(R^{n}+\phi_{n}(R,r,\alpha,\beta)r^{n})\Lambda_{n}|e^{i\gamma}+|(1+\phi_{n}(R,r,\alpha,\beta))\lambda_{0}|\Big|^{p}d\gamma\int_{0}^{2\pi}\left|P(e^{i\theta})\right|^{p}d\theta,\\
  &=\int\limits_{0}^{2\pi}\Big|(R^{n}+\phi_{n}(R,r,\alpha,\beta)r^{n})\Lambda_{n}e^{i\gamma}+(1+\phi_{n}(R,r,\alpha,\beta))\lambda_{0}\Big|^{p}d\gamma\int_{0}^{2\pi}\left|P(e^{i\theta})\right|^{p}d\theta,
  \end{align}
  the desired result follows immediately by combining \eqref{r7} and \eqref{r8}. This completes the proof of Theorem \ref{t2} for $p>0$. To establish this result for $p=0$, we simply let $p\rightarrow 0+$.

   \end{proof}


\begin{thebibliography}{99}
\bibitem{1} N.C. Ankeny and T.J.Rivli, On a theorm of S.Bernstein, Pacific J. Math., \textbf{5}(1955), 849 - 852.
\bibitem{2} V.V. Arestov, On integral inequalities for trigonometric polynimials and their derivatives, Izv. Akad. Nauk SSSR Ser. Mat. \textbf{45 }(1981),3-22[in Russian]. English translation; Math.USSR-Izv.,\textbf{18} (1982), 1-17. 
\bibitem{3} A. Aziz, A new proof and a generalization of a theorem of De Bruijn, proc. Amer Math. Soc., \textbf{106}(1989), 345-350.
\bibitem{ar97} A. Aziz and N.A. Rather, $L^p$ inequalities for polynomials,Glas. Math., \textbf{32} (1997), 39-43.
\bibitem{4} A. Aziz and N.A. Rather, Some new generalizations of Zygmund-type inequalities for polynomials, Math. Ineq. Appl., \textbf{15}(2012), 469-486. 
\bibitem{5}R.P. Boas, Jr. and Q.I. Rahman , $L^{p}$ inequalities for polynomials and entire functions, Arch. Rational Mech. Anal., \textbf{11}(1962),34-39. 
\bibitem{6} N.G. Bruijn, Inequalities concerning polynomials in the complex domain, Nederal. Akad.Wetensch. Proc.,\textbf{50}(1947), 1265-1272. 
\bibitem{8} G.H. Hardy, The mean value of the modulus of an analytic functions, Proc. London Math. Soc., \textbf{14}(1915), 269-277.
\bibitem{9} P.D. Lax, Proof of a conjecture of P.Erd\"{o}s on the derivative of a polynomial, Bull. Amer. Math.Soc.,\textbf{50}(1944), 509-513.
\bibitem{10} M. Marden, Geometry of polynomials, Math. Surveys, No.\textbf{3}, Amer. Math.Soc. Providence, RI, 1949.
\bibitem{11} G.V. Milovanovic, D.S. Mitrinovic and Th.M. Rassias, Topics in Polynomials: Extremal Properties, Inequalities, Zeros, World scientific Publishing Co., Singapore,(1994).
\bibitem{12} G. Polya and G. Szeg\"{o}, Aufgaben und lehrs\"{a}tze aus der analysis, Springer-Verlag, Berlin (1925).
\bibitem{13} Q.I. Rahman, Functions of exponential type, Trans. Amer. Math. Soc., \textbf{ 135}(1969), 295-309. 
\bibitem{14} Q.I. Rahman and G. Schmeisser, $L^{p}$ inequalities for polynomials, J. Approx. Theory, \textbf{53}(1988), 26-32.
\bibitem{15} Q.I. Rahman and G. Schmeisser, Analytic Theory of Polynomials, Oxford University Press, New York, 2002.
\bibitem{rs} N.A. Rather and M.A. Shah, On an operator preserving $L_p$ inequalities between polynomials, \textbf{399} (2013), 422-432.
\bibitem{rsg} N.A. Rather and Suhail Gulzar, Integral mean estimates for an operator preserving inequalities between polynomials, J. Inequal. Spec. Funct., \textbf{3} (2012),  24 - 41.
\bibitem{16}  A.C. Schaffer, Inequalities of A.Markov and S.Bernstein for polynomials and related functions, Bull.Amer. Math. Soc., \textbf{47}(1941), 565-579.
\bibitem{17} W.M.Shah and A.Liman, Integral estimates for the family of B-operators, Operators and Matrices, \textbf{5}(2011), 79-87.
\bibitem{18} A. Zygmund, A  remark on conjugate series, Proc. London Math. Soc., \textbf{34}(1932),292-400.
\end{thebibliography}
\end{document}